\documentclass[10pt,twoside]{IEEEtran}
\usepackage {amssymb,rotating,graphicx,cite, calc, color, psfrag}
\usepackage{authblk}

\usepackage[cmex10]{amsmath}
\usepackage{amssymb,amsthm,amsfonts,mathrsfs}
\usepackage{graphicx}
\usepackage{enumitem,stackrel}
\usepackage{color}
\usepackage{mathtools,amssymb,bm}
\usepackage{amstext}
\usepackage{array}
\usepackage{algorithmicx}
\usepackage{hyperref}
\theoremstyle{remark}

\ifCLASSOPTIONcaptionsoff
  \usepackage[nomarkers]{endfloat}
 \let\MYoriglatexcaption\caption
 \renewcommand{\caption}[2][\relax]{\MYoriglatexcaption[#2]{#2}}
\fi
\newcommand{\RN}[1]{%
	\textup{\uppercase\expandafter{\romannumeral#1}}%
}

\newtheoremstyle{mystyle}
  {}
  {}
  {\itshape}
  {}
  {\bfseries}
  {.}
  { }
  {}
\theoremstyle{mystyle}

{}
\newtheorem{example}{Example}{}
{}

\newtheorem{defi}{Definition}{}
{}
\usepackage{epstopdf}
\newtheorem{theorem}{Theorem}

\begin{document}
%
\title{Haplotype Assembly Using Manifold Optimization and Error Correction Mechanism}
\author{Mohamad Mahdi~Mohades, Sina~Majidian, and Mohammad Hossein~Kahaei \thanks{The authors are with the School of Electrical Engineering, Iran University of Science \& Technology, Tehran 16846-13114, Iran (e-mail: mohamad\_mohaddes@elec.iust.ac.ir; s\_majidian@elec.iust.ac.ir; kahaei@iust.ac.ir).}}%

\maketitle

\begin{abstract}
Recent matrix completion based methods have not been able to properly model  the Haplotype Assembly  Problem (HAP)  for noisy observations. To deal with such cases, we propose a new Minimum Error Correction (MEC) based matrix completion problem over the manifold of rank-one matrices. We then prove the convergence of a specific iterative algorithm to solve this problem. From the simulation results, the proposed method not only outperforms some well-known  matrix completion based methods, but also shows a more accurate result compared to  a most recent MEC based algorithm for haplotype estimation.
\end{abstract}


%
\IEEEpeerreviewmaketitle

\section{Introduction}\label{sec:Introduction}
\IEEEPARstart{S}{tudy} of genetic variations is a critical task to disease recognition and also drug development. Such variations can be recognized using haplotypes which are known as strings of Single Nucleotide Polymorphisms (SNPs) of chromosomes \cite{Schwartz_2010}. However, due to hardware limitations, haplotypes cannot be entirely measured. Instead, they are supposed to be reconstructed from the short reads obtained by high-throughput sequencing systems.
For diploid organisms, each SNP site is modeled by either $+1$ or $-1$. Hence, a haplotype sequence $\bf{h}$ corresponding to one of chromosomes is formed by a vector of $\pm 1$ entries, while the haplotype of the other one is given by  $-\bf{h}$. Accordingly, the short reads are obtained from either $\bf{h}$ or $-\bf{h}$.
 Such a sampling procedure can be modeled by a matrix, say ${\bf{R}}$, whose entries are either $\pm 1$ or unobserved. Specifically, if we denote the set of observed entries by $\Omega$, the entries of ${\bf{R}}\in{\left\{-1,+1,\times\right\}}^{m\times n}$ are given by
\begin{equation} \label{Elements_of_Main_Matrix}
\left\{ {\begin{array}{*{20}{c}}
{{r_{ij}} =  \pm 1\,\,\,\,\,\,\left( {i,j} \right) \in \Omega }\\
{{r_{ij}} =  \times \,\,\,\,\,\,\left( {i,j} \right) \notin \Omega , }
\end{array}} \right.
\end{equation}
where $\times$ shows unobserved entries.
The completed matrix of $\bf{R}$, defined by $\overline{\bf{R}}$, is supposed to be a rank-one matrix whose entries are $\pm1$ and can be factorized as
\begin{equation} \label{Factorization_of_Main_Matrix}
\overline{\bf{R}}={\bf{c}}_{m\times 1}{\bf{h}}_{n\times 1}^T,
\end{equation}
where ${\bf{h}}$ is the intended haplotype sequence and {\bf{c}} is a sequence of $+1$  and  $-1$ entries corresponding to ${\bf{h}}$ and $-{\bf{h}}$, respectively. When the number of reads (observed entries)  are sufficient and there is no error in reads, it is easy to obtain $\overline{\bf{R}}$ and consequently $\bf{h}$.

In practice, however, measurements are corrupted by noise which usually leads to erroneous signs in some observed entries.
Many haplotype assembly methods deal with erroneous measurements  mostly based on the MEC defined as \cite{HapCUT}
\begin{equation} \label{MEC_formulation}
{\rm{MEC}}\left( {{\bf{R}},{\bf{h}}} \right) = \sum\limits_{i = 1}^m {\min \left( {hd\left( {{{\bf{r}}_i},{\bf{h}}} \right),hd\left( {{{\bf{r}}_i}, - {\bf{h}}} \right)} \right)},
\end{equation}
where the Hamming distance $hd(\cdot,\cdot)$ calculates the number of mismatch entries using
\begin{equation} \label{Hamming_distance}
hd\left( {{{\bf{r}}_i},{\bf{h}}} \right) = \sum\limits_{\left. j \right|\left( {i,j} \right) \in \Omega }^{} {d\left( {{r_{ij}},{h_j}} \right)},
\end{equation}
with   $d(\cdot,\cdot)$ being $0$ for equal inputs, and $1$ otherwise.

For optimal  solution of (\ref{MEC_formulation}), which is an NP-hard problem, some heuristic methods have already  been addressed 
\cite{HapCUT,Puljiz_2016,hashemi2018sparse}. Apart from MEC approaches, some other methods are developed based on the fact that the rank of $\overline{\bf{R}}$  ought to be 1. In \cite{Changxiao_Cai},  matrix factorization is used in the minimization problem,
\begin{equation} \label{Frobenius_factorization}
\mathop {\min }\limits_{{\bf{u}},{\bf{v}}}  \left\| {{{\rm{P}}_\Omega }\left( {\bf{R}} \right) - {{\rm{P}}_\Omega }\left( {{\bf{u}}_{m\times1}{{\bf{v}}_{n\times1}^T}} \right)} \right\|_F^2,
\end{equation}
where ${{\rm{P}}_\Omega }(\cdot)$ is a sampling operator  defined as
\begin{equation} \label{P_Omega}
{{\rm{P}}_\Omega }({\bf{Q}})=\left\{ {\begin{array}{*{20}{c}}
{{{\rm{P}}_\Omega }({q_{ij})} = {q_{ij}} \,\,\,\,\,\,\left( {i,j} \right) \in \Omega }\\
{{{\rm{P}}_\Omega }({q_{ij})} =  0 \,\,\,\,\,\,\left( {i,j} \right) \notin \Omega }
\end{array}} \right.
\end{equation}
and $\|\cdot\|_F$ shows the Frobenius norm.
Then,  $\bf{h}$ is calculated by applying the sign function over $\bf{v}$. The nonconvex problem of (\ref{Frobenius_factorization}) can be solved using the alternating minimization approach. For fast matrix completion, QR decomposition may also be applied \cite{liu2018fast}.

From the low-rank matrix completion point of view, haplotype estimation can be performed using some other approaches as follows.
Based on \cite{Candes_Recht}, it is possible to complete ${{\rm{P}}_\Omega }\left( {\bf{R}} \right)$ using the following optimization problem,

\begin{equation} \label{Nuclear_norm}
\mathop {\min }\limits_{\bf{X}} \,\,\,\left\| {{{\rm{P}}_\Omega }\left( {\bf{R}} \right) - {{\rm{P}}_\Omega }\left( {\bf{X}} \right)} \right\|_F^2 + \lambda {\left\| {\bf{X}} \right\|_*},
\end{equation}
where ${\left\| \cdot \right\|_*}$ shows the nuclear norm and $\lambda$ is a regularization factor. Afterwards, we can estimate $\bf{h}$ by applying the sign function over the right singular vector corresponding to the largest singular value.
In addition, minimizing the cost function $\left\| {{{\rm{P}}_\Omega }\left( {\bf{R}} \right) - {{\rm{P}}_\Omega }\left( {\bf{X}} \right)} \right\|_F^2$ over the manifold of rank-one matrices \cite{Bart} or Grassmann manifold \cite{Keshavan_Few_Entries} will result in completion of ${{\rm{P}}_\Omega }\left( {\bf{R}} \right)$ and estimating $\bf{h}$.

In this paper, we propose a new error correction mechanism over the manifold of rank one matrices to solve the noisy HAP. This approach benefits the underlying structure of  the HAP,  which seeks for a rank one matrix with the entries of $\pm 1$. As a result, unlike common matrix completion methods, we propose a Hamming distance cost function over the manifold of rank one matrices to complete the desired matrix. We next present a surrogate for the nondifferentiable structure of this cost function and analyze its convergence behaviour. In this way, as opposed to the existing MEC based algorithms, we will be able to derive the performance guarantees for our optimization problem.
Simulation results confirm that this method is more reliable for solving noisy HAPs.

The paper is organized as follows. In Section \ref{sec:Preliminary} some required concepts of manifold optimization are presented. Section \ref{Main_results} is devoted to our main result. Simulation results are illustrated in Section \ref{Simulation} and Section \ref{section.conclusion} concludes the paper.


\section{Preliminaries}\label{sec:Preliminary}
As mentioned, a HAP can be modeled as a rank-one matrix completion problem. On the other hand, the set of real valued rank-one matrices can be considered as a smooth manifold \cite{Bart}. To optimize a differentiable function over such a manifold, we first remind some concepts for optimization over manifolds.

\begin{theorem} \cite{Bart} \label{manifold_CRM}
The set of all $m\times n$ real valued matrices of rank $1$ is a smooth manifold, $\mathcal{M}^{(1)}$, whose tangent space at point ${\bf{X}}\in\mathcal{M}^{(1)}$ is  defined as

\begin{equation} \label{tangent_space_CRM}
\resizebox{.99 \hsize}{!}{$
\begin{array}{l}
{T_{\bf{X}}}{{\mathcal M}^{(1)}} = \left\{ {\left[ {{\bf{U}}\,\,{{\bf{U}}_ \bot }} \right]\left[ {\begin{array}{*{20}{c}}
{{\mathbb{R}}}&{{\mathbb{R}^{1 \times \left( {n - r} \right)}}}\\
{{\mathbb{R}^{\left( {m - 1} \right) \times 1}}}&{{{\bf{0}}_{\left( {m - 1} \right) \times \left( {n - 1} \right)}}}
\end{array}} \right]{{\left[ {{\bf{V}}\,\,{{\bf{V}}_ \bot }} \right]}^T}} \right\}\\
\,\,\,\,\,\,\,\,\,\,\,\,\, = \left\{ {{\bf{UM}}{{\bf{V}}^T} + {{\bf{U}}_p}{{\bf{V}}^T} + {\bf{UV}}_p^T:} \right.{\bf{M}} \in {\mathbb{R}},\\
\,\,\,\,\,\,\,\,\,\,\,\,\,\,\,\,\,\,\,\,\,\,\,\,{{\bf{U}}_p} \in {\mathbb{R}^{m \times 1}},{\bf{U}}_p^T{\bf{U}} = {\bf{0}},{{\bf{V}}_p} \in {\mathbb{R}^{n \times 1}},{\bf{V}}_p^T{\bf{V}} = {\bf{0}},
\end{array} $}
\end{equation}
where ${\bf{U\Sigma V}}^T$ is the singular value decomposition (SVD) of ${\bf{X}}$.
\end{theorem}

Convergence of manifold optimization schemes is evaluated by a metric on the manifold. Inspired from the Euclidean space $\mathbb{R}^{m\times n}$, the metric on the manifold of rank-one matrices is defined by the inner product $\left<\xi , \eta\right>_{\bf{X}}={\rm{tr}}\left( {{\xi^T}\eta} \right)$, where the subscript $\bf{X}$ shows the restriction to the tangent space $T_{\bf{X}}\mathcal{M}^{(1)}$ and the tangent vectors $\xi,\eta\in T_{\bf{X}}\mathcal{M}^{(1)}$.
A smoothly varying inner product is Riemannian metric and the manifold endowed with this metric is called Riemannian. Also, the gradient of a function on such a manifold is given as follows.
\begin{defi} \label{gradient} \cite{Absil_Book}
Let $f$ be a scalar valued function over the Riemannian manifold $\mathcal{M}$ endowed with the inner product $\left<\cdot , \cdot\right>_{\bf{X}}$. Then, the gradient of $f$ at point ${\bf{X}}\in\mathcal{M}$ denoted by ${\rm{grad}}f({\bf{X}})$ is the unique element of the tangent space  $T_{\bf{X}}\mathcal{M}$ satisfying
\begin {equation}
\left<{\rm{grad}}f({\bf{X}}),\xi\right>_{\bf{X}}={\rm{D}}f({\bf{X}})[\xi] \,\,\,\,\,\,  \forall \xi \in T_{\bf{X}}\mathcal{M},
\end {equation}
where ${\rm{D}}f$ denotes the directional derivative acting on the tangent vector $\xi$.
\end{defi}

By iteratively solving an optimization problem over the manifold $\mathcal{M}$, we may find a point on $T_{\bf{X}}\mathcal{M}$ which would not necessarily belong to the manifold. To bring this point back to the manifold, the retraction function can be utilized.

\begin{defi}\cite{Absil_Book}\label{retraction_defi}
Let $T\mathcal{M}: = \mathop  \cup \limits_{{\bf{X}} \in {{\mathcal M}}} \left\{ {\bf{X}} \right\} \times {T_{\bf{x}}}{{\mathcal M}}$ be the tangent bundle and $R:T\mathcal{M}\to\mathcal{M}$ be a smooth mapping whose restriction to $T_{\bf{X}}\mathcal{M}$ is $R_{\bf{X}}$. Then, $R$ is a retraction on the manifold $\mathcal{M}$, if,\\
$(1)$ $R_{\bf{X}}({\bf{0}}_{\bf{X}}) = {\bf{X}}$, where, ${\bf{0}}_{\bf{X}}$ is the zero element of $T_{\bf{X}}\mathcal{M}$,
\\$(2)$ ${\rm{D}}R_{\bf{X}}({\bf{0}}_{\bf{X}}) = {\rm{id}}_{T_{\bf{X}}\mathcal{M}}$, where, ${\rm{id}}_{T_{\bf{X}}\mathcal{M}}$ is the identity mapping on ${T_{\bf{X}}\mathcal{M}}$.
\end{defi}
 More clearly, we  find the result of $R_{\bf{X}}(\xi)$ by first calculating  ${\bf{Y}}={\bf{X}}+\xi$ where ${\bf{X}}\in\mathcal{M}$ and $\xi\in T_{\bf{X}}\mathcal{M}$,  and next applying the retraction mapping on ${\bf{Y}}$.
It has been shown that the retraction $R_{\bf{X}}(\xi)$ to the manifold of rank-one matrices is equivalent to taking the SVD of  ${\bf{Y}}$  and then making all the singular values equal to zero except the largest one  \cite{Absil_Paper}.

Another concept to note  is the gradient descent algorithm on Riemannian manifolds which is given in Alg. 1 \cite{Absil_Book}. As seen, Steps 1 and 2 concern with the search direction and convergence evaluation. Step 3, known as Armijo backtracking, guarantees a sequence of points which constitutes a descent direction \cite{Absil_Book}. Step 4 performs retraction on the manifold.

\begin{center} \label{Line_Search_Absil}
 \begin{tabular}{|c|}
 \hline
\hspace{-.2cm}\textbf{Alg. 1}: Gradient descent method on a Riemannian manifold. \\
 \hline
\hspace{-.1cm}\textbf{Requirements}: Differentiable cost function $f$, Manifold $\mathcal{M}$,\\ \hspace{-.6cm}inner product $\left<\cdot,\cdot\right>$, Initial matrix ${\bf{X}}_0 \in \mathcal{M}$, Retraction \\ function  $R$, Scalars $\bar{\alpha}>0, \,  \, \beta,\sigma\in(0,1)$, tolerance $\tau>0$.\\
\hspace{-5.7cm}{\bf{for}}  $i=0,1,2,...$  {\bf{do}}\\
\hspace{-.4cm}\textbf{Step 1}: \,\,\,\, Set $\xi$ as the negative direction of the gradient, \\
 \hspace{-2.6cm}$\xi_i:=-{\rm{grad}}f({\bf{X}}_i)$\\
\hspace{-3.8cm}\textbf{Step 2}:\,\,\,\, Convergence evaluation,\\
\hspace{-2cm}\,\,\,\, \textbf{if} $\left\|\xi_i\right\| < \tau$, \textbf{then break}\\
\hspace{-3.00cm}\textbf{Step 3}:\,\,\,\, Find the smallest $m$ satisfying\\
\,\,\,\,\,\,\,\, $f({\bf{X}}_i)-f(R_{{\bf{X}}_i}(\bar{\alpha}\beta^m\xi_i))\geq \sigma\bar{\alpha}\beta^m \left<\xi_i,\xi_i\right>_{{\bf{X}}_i}$\\
\hspace{-3.2cm}\textbf{Step 4}:\,\,\,\,\,\,\,\,\, Find the modified point as \\
\,\,\,\,\,\,\,\,\,\, ${\bf{X}}_{i+1}:=R_{{\bf{X}}_i}(\bar{\alpha}\beta^m\xi_i))$\\
\hline
\end{tabular}
\end{center}
\section {Main result} \label{Main_results}
First, using an example for the noisy HAP, we show that minimizing the cost function $\left\| {{{\rm{P}}_\Omega }\left( {\bf{R}} \right) - {{\rm{P}}_\Omega }\left( {\bf{X}} \right)} \right\|_F^2$ and calculating the haplotype by applying the sign function over the right singular vector corresponding to the largest singular value would not lead to the desired result. Then, we propose an optimization problem which can properly model the HAP.

\begin{example}\label{Fro_Norm_Inconsistent}
Substitute  ${\bf{h}} = \left[ {\begin{array}{*{20}{c}}{1}&{-1}&1&{-1}&{-1}\end{array}} \right]^T$ and ${\bf{c}} = \left[ {\begin{array}{*{20}{c}}{1}&
{1}&1\end{array}} \right]^T$ in $\rm{(}$\ref{Factorization_of_Main_Matrix}$\rm{)}$ to build up $\overline{\bf{R}}$ as
\begin {equation}
\overline{\bf{R}} = \left[ {\begin{array}{*{20}{c}}
{\begin{array}{*{20}{c}}
{  1}\\
{ 1}\\
{ 1}
\end{array}}&{\begin{array}{*{20}{c}}
{ - 1}\\
{ - 1}\\
{ - 1}
\end{array}}&{\begin{array}{*{20}{c}}
1\\
{  1}\\
1
\end{array}}&{\begin{array}{*{20}{c}}
{ - 1}\\
{ - 1}\\
{ - 1}
\end{array}}&{\begin{array}{*{20}{c}}
{ - 1}\\
{ - 1}\\
{ - 1}
\end{array}}
\end{array}} \right].\end{equation}
Next, let an erroneous sampling from $\overline{\bf{R}}$ be given as
\begin{equation}
{{\rm{P}}_\Omega }\left( {{{\overline {\bf{R}} }_E}} \right) = \left[ {\begin{array}{*{20}{c}}
{\begin{array}{*{20}{c}}
{  1}\\
{ -1}\\
{ 1}
\end{array}}&{\begin{array}{*{20}{c}}
{ 0}\\
{ - 1}\\
{ - 1}
\end{array}}&{\begin{array}{*{20}{c}}
1\\
{  1}\\
{-1}
\end{array}}&{\begin{array}{*{20}{c}}
{ - 1}\\
{ - 1}\\
{ - 1}
\end{array}}&{\begin{array}{*{20}{c}}
{  1}\\
{ - 1}\\
{ - 1}
\end{array}}
\end{array}} \right],\end{equation}
where subscript $E$ stands for erroneous observation. As seen, the only unobserved entry is ${\left\{{\left(1,2\right)}\right\}}$ and the set of erroneous observations is ${\left\{{\left(1,5\right)},{\left(2,1\right)},{\left(3,3\right)}\right\}}$.
Based on our simulations, minimization of $\left\| {{{\rm{P}}_\Omega }\left( \overline{\bf{R}}_E \right) - {{\rm{P}}_\Omega }\left( {\bf{X}} \right)} \right\|_F^2$ over the manifold of rank-one matrices could lead to either
\begin {equation}
{\overline {\bf{R}} _1} =\left[ {\begin{array}{*{20}{c}}
{0.2255}&{ - 0.763}&{0.2255}&{-0.8165}&{ - 0.3655}\\
{0.2955}&{ - 1}&{0.2955}&{ - 1.07}&{ - 0.{\rm{4790}}}\\
{0.2955}&{ - 1}&{0.2955}&{ - 1.07}&{ - 0.{\rm{4790}}}
\end{array}} \right]
\end{equation}
or
\begin {equation}
{\overline {\bf{R}} _2} = \left[ {\begin{array}{*{20}{c}}
{ - \varepsilon \gamma }&{{\rm{0}}{\rm{.9975}}\varepsilon }&{\varepsilon \gamma }&{{\rm{0}}{\rm{.9975}}\varepsilon }&{{\rm{0}}{\rm{.9975}}\varepsilon }\\
\gamma &{{\rm{ - 0}}{\rm{.9975}}}&{ - \gamma }&{{\rm{ - 0}}{\rm{.9975}}}&{{\rm{ - 0}}{\rm{.9975}}}\\
\gamma &{{\rm{ - 0}}{\rm{.9975}}}&{ - \gamma }&{{\rm{ - 0}}{\rm{.9975}}}&{{\rm{ - 0}}{\rm{.9975}}}
\end{array}} \right],
\end{equation}
where $\varepsilon$ and $\gamma$ are infinitesimal positive numbers. However, one can easily check that although for this example we get  \\$\left\| {{{\rm{P}}_\Omega }\left( \overline{\bf{R}}_E \right) - {{\rm{P}}_\Omega }\left( {\overline {\bf{R}} _1} \right)} \right\|_F\simeq\left\| {{{\rm{P}}_\Omega }\left( \overline{\bf{R}}_E \right) - {{\rm{P}}_\Omega }\left( {\overline {\bf{R}} _2} \right)} \right\|_F\simeq2.8284$, their haplotypes are estimated differently as ${\bf{h}}_1 = \left[ {\begin{array}{*{20}{c}}{ 1}&{-1}&{  1}&{-1}&{-1}\end{array}} \right]^T$ and ${\bf{h}}_2 = \left[ {\begin{array}{*{20}{c}}{ 1}&{-1}&{-1}&{ - 1}&{-1}\end{array}} \right]^T$, while only
${\bf{h}}_1$ is the correct answer. Note that haplotypes are estimated by applying the sign function over the right singular vector of the completed rank-one matrix.
\end{example}
For the above example, having the Hamming distance in mind, it is easy to verify that the following inequality holds,
\begin{equation}\label{Compare_Fro}
\begin{array}{l}
hd\left( {vec\left(  {{{\rm{P}}_\Omega }\left({sign\left( {{{\overline {\bf{R}} }_1}} \right)} \right)} \right),vec\left( {\left( {{{\rm{P}}_\Omega }\left( {{{\overline {\bf{R}} }_E}} \right)} \right)} \right)} \right) < \\
hd\left( {vec\left(  {{{\rm{P}}_\Omega }\left( {sign\left({{{\overline {\bf{R}} }_2}} \right)} \right)} \right),vec\left( {\left( {{{\rm{P}}_\Omega }\left( {{{\overline {\bf{R}} }_E}} \right)} \right)} \right)} \right),
\end{array}\end{equation}
where $vec\left(\cdot\right)$ vectorizes the input matrix. Equivalently, we have
\begin{equation}\label{Compare_Fro_Matrix}
\begin{array}{l}
{\left\| {{{{\rm{P}}_\Omega }\left({sign\left( {{{\overline {\bf{R}} }_1}} \right)} \right)}-{{\rm{P}}_\Omega }\left( {{{\overline {\bf{R}} }_E}} \right)} \right\|_0 <}\\
{\left\| {{{{\rm{P}}_\Omega }\left({sign\left( {{{\overline {\bf{R}} }_2}} \right)} \right)}-{{\rm{P}}_\Omega }\left( {{{\overline {\bf{R}} }_E}} \right)} \right\|_0 },
\end{array}
\end{equation}
where $\|\cdot\|_0$ shows the $l_0$-norm  which counts the number of nonzero entries of a matrix.

Therefore, we propose the following minimization cost function for noisy HAPs as
\begin{equation} \label {HAP_MMM_0}
\mathop {\min }\limits_{{\bf{X}} \in \mathcal{M}^{(1)}} {\left\| {{{\rm{P}}_\Omega }\left( {{{\overline {\bf{R}} }_E}} \right) - {{\rm{P}}_\Omega }\left( {{\rm{sign}}\left( {\bf{X}} \right)} \right)} \right\|_0}.
\end{equation}
    Using (\ref{HAP_MMM_0}), we can find a rank-one matrix  whose sign is as similar as possible to the observations. However,  since (\ref{HAP_MMM_0}) is nondifferentiable, we replace ${\rm{sign}}(x_{ij})$ by the differentiable function $\gamma_1{\tan ^{ - 1}}(\gamma_2 x_{ij})$ where the positive values $\gamma_1$ and $\gamma_2$ are selected to approximate the sign function.
Another difficulty with nondifferentiability of (\ref{HAP_MMM_0}) is the  discontinuity of $l_0$-norm  which is solved by replacing the $l_p$-norm $\|\cdot\|_p$ that is applied to the vector form of the matrix for $p>0$.
However, since $l_p$-norm is yet nondifferentiable in the case of
\begin{equation} \label {HAP_MMM_Condition}
{\left({{\rm{P}}_\Omega }\left( {{{\overline {\bf{R}} }_E}} \right) - {{\rm{P}}_\Omega }\left( {\gamma_1{\tan ^{ - 1}}\left( \gamma_2{\bf{X}} \right)} \right)\right)_{ij}}=0
\end{equation}
for any $\left( {i,j} \right) \in \Omega$, we limit the value of $\gamma_1$ to $\left(0,{\raise0.7ex\hbox{$2$} \!\mathord{\left/
 {\vphantom {2 \pi }}\right.\kern-\nulldelimiterspace}
\!\lower0.7ex\hbox{$\pi $}}\right)$ . In this way, ({\ref{HAP_MMM_0}}) is reformulated in a new problem as
\begin{equation} \label {HAP_MMM_p}
\mathop {\min }\limits_{{\bf{X}} \in \mathcal{M}^{(1)}} {f({\bf{X}})=\left\| {{{\rm{P}}_\Omega }\left( {{{\overline {\bf{R}} }_E}} \right) - {{\rm{P}}_\Omega }\left( {\gamma_1{\tan ^{ - 1}}\left( \gamma_2{\bf{X}} \right)} \right)} \right\|_p^p},
\end{equation}
where  ${\tan ^{ - 1}}(\cdot)$ acts elementwise.
 Although, for $0< p\leq1$ the sparsity is promoted in  (\ref{HAP_MMM_p}), here we choose $p$ equal or slightly greater than 1 to easily use the  triangle inequality to prove the convergence of Alg. 1. Simulation results will show that such a choice will lead to more accurate estimates of haplotypes. Now, we prove the convergence of Alg. 1 for (\ref{HAP_MMM_p}).
\begin {theorem} \label {Convergence}
Let positive values $\gamma_1<{\raise0.7ex\hbox{$2$} \!\mathord{\left/
 {\vphantom {2 \pi }}\right.\kern-\nulldelimiterspace}
\!\lower0.7ex\hbox{$\pi $}}$ and $\gamma_2$ be given. Then, by choosing the initial matrix ${\bf{X}}_0$ in such a way that $\|{\bf{X}}_0\|_F^2$ is bounded and $f({\bf{X}}_0)<|\Omega|$, Alg. 1 will converge.
\end{theorem}
\begin{proof}

Using Theorem 4.3.1 in \cite{Absil_Book}, we know that every limit  point of the infinite sequence $\{{\bf{X}}_i\}$ generated by Alg. 1 is a critical point of $f(\bf{X})$   in (\ref{HAP_MMM_p}). Accordingly, we need to prove that under the mentioned conditions of Theorem \ref{Convergence}, $\{{\bf{X}}_i\}$ owns a limit point and also belongs to the manifold of rank-one matrices, meaning that it converges. To do so, we show that this sequence lies in a compact set in which any sequence owns a convergent subsequence \cite{kreyszig1978introductory}. Consequently, due to the fact that $\{{\bf{X}}_i\}$ is a decreasing sequence, having a convergent subsequence guarantees its convergence.

To show the compactness of the set containing $\{{\bf{X}}_i\}$, we need to prove that the set is bounded and closed \cite{kreyszig1978introductory}. First, we prove the closedness of the set. Suppose that in our problem there exists a limit point for $\{{\bf{X}}_i\}$, which can be either a zero matrix, or a rank-one matrix belonging to the manifold of rank-one matrices. However, we ought to show that under the conditions of Theorem \ref{Convergence} the limit point of $\{{\bf{X}}_i\}$ is a rank-one matrix. To see this, it is enough to prove that the zero matrix is not a limit point of $\{{\bf{X}}_i\}$. By Step 3 of  Alg. 1, we know that  $\{{\bf{X}}_i\}$ is a decreasing sequence for $f({\bf{X}})$. Suppose that $\{{\bf{Y}}_j\}$ is a subsequence of  $\{{\bf{X}}_i\}$ converging to the zero matrix. Then, there is an integer $K$ such that $\forall j>K$, $\|{\bf{Y}}_j\|_p^p\leq\|{\bf{Y}}_j\|_F^2<2^{-j}$.
From the triangle inequality, we get $f({\bf{Y}}_j)\geq\mathop {\left\| {{{\rm{P}}_\Omega }\left( {{{\overline {\bf{R}} }_E}} \right)} \right\|_p^p}-{\left\| {{\rm{P}}_\Omega }\left( {\gamma_1{\tan ^{ - 1}}\left(\gamma_2{\bf{Y}}_j \right)} \right)\right\|_p^p}$, $\forall j>K$, and thus,  $\mathop {\lim }\limits_{j \to \infty } f({{\bf{Y}}_j}) \geq \left| \Omega  \right|$.
Now, we easily enforce  Alg. 1 to have an initial matrix $\{{\bf{X}}_0\}$ for which $f({\bf{X}}_0) < \left|\Omega\right|$. Hence, due to the fact that $\{{\bf{X}}_i\}$ is  a decreasing sequence for $f({\bf{X}})$, the zero matrix will not be the limit point. This ensures that $\{{\bf{X}}_i\}$ is in a closed set.

Now, we show that $\{{\bf{X}}_i\}$ stays in a bounded set. For this to happen, we modify (\ref{HAP_MMM_p}) to define $f_{\rm{m}}(\bf{X})$  by adding the regularization term $\mu \left\|{\bf{X}}\right\|_F^2$, where $\mu$ is a positive value. Since we enforced Alg. 1 to have the initial matrix ${\bf{X}}_0$ satisfying $f({\bf{X}}_0) < \left|\Omega\right|$,  we get $f_{\rm{m}}({\bf{X}}_0) < \left|\Omega\right|+\mu\left\|{\bf{X}}_0\right\|_F^2$. Also, let  ${\bf{X}}_0$ be a norm bounded initial matrix. Consequently, $f_{\rm{m}}({\bf{X}}_0)$ will be bounded to a value, say $z$.  Therefore, the decreasing sequence $\{{\bf{X}}_i\}$ generated by Alg. 1 for $f_{\rm{m}}$ satisfies $f({\bf{X}}_i)+\mu\left\|{\bf{X}}_i\right\|_F^2 < z$ showing that $\left\|{\bf{X}}_i\right\|_F^2$ is bounded. Considering the fact that by choosing an infinitesimal value for $\mu$, the sequence generated by Alg. 1 for either  $f({\bf{X}})$ or $f_{\rm{m}}({\bf{X}})$ is the same,   $\{{\bf{X}}_i\}$ stays in a bounded set.

From the above reasoning,   $\{{\bf{X}}_i\}$ stays in a compact set, meaning that the proposed problem in (\ref{HAP_MMM_p}) will converge.
\end {proof}
Note that although replacing ${\rm{sign}}\left( {\bf{X}} \right)$  by ${\gamma _1}{\tan ^{ - 1}}\left( {{\gamma _2}{\bf{X}}} \right)$  in general leads to slightly reducing the haplotype estimation accuracy,   in effect, it makes the optimization problem mathematically tractable and also practical.
\section {Simulation Results} \label{Simulation}
For comparison purposes two criteria are considered. First, we evaluate the completion performance of different methods using the Normalized Mean Square Error (NMSE) defined as
\begin{equation} \label{NMSE}
\begin{array}{l}
{\rm{NMSE}} = \frac{{\left\| {{\overline{\bf{R}}} - \widehat {{\overline{\bf{R}}}}} \right\|_F^2}}{{\left\| {\overline{\bf{R}}} \right\|_F^2}}
\end{array},
 \end{equation}
where $\widehat {{\overline{\bf{R}}}}$ is the estimation of the original matrix ${\overline{\bf{R}}}$.
The other criterion is the Hamming distance ($\rm{hd}$) of the original haplotype and its estimate. Note that since the original haplotype can be either $\bf{h}$ or $-\bf{h}$, we calculate the Hamming distance of the estimated haplotype for  both cases and choose the minimum one.
Our results are compared with the matrix completion based methods  \cite{Changxiao_Cai},\cite{Bart},\cite{Keshavan_Few_Entries},\cite{cai2010singular} and the most recent MEC based algorithm \cite{hashemi2018sparse}.
\begin{figure}[!ht]
\centering
\begin{minipage}[b]{0.365\textwidth}
    \includegraphics[width=\textwidth]{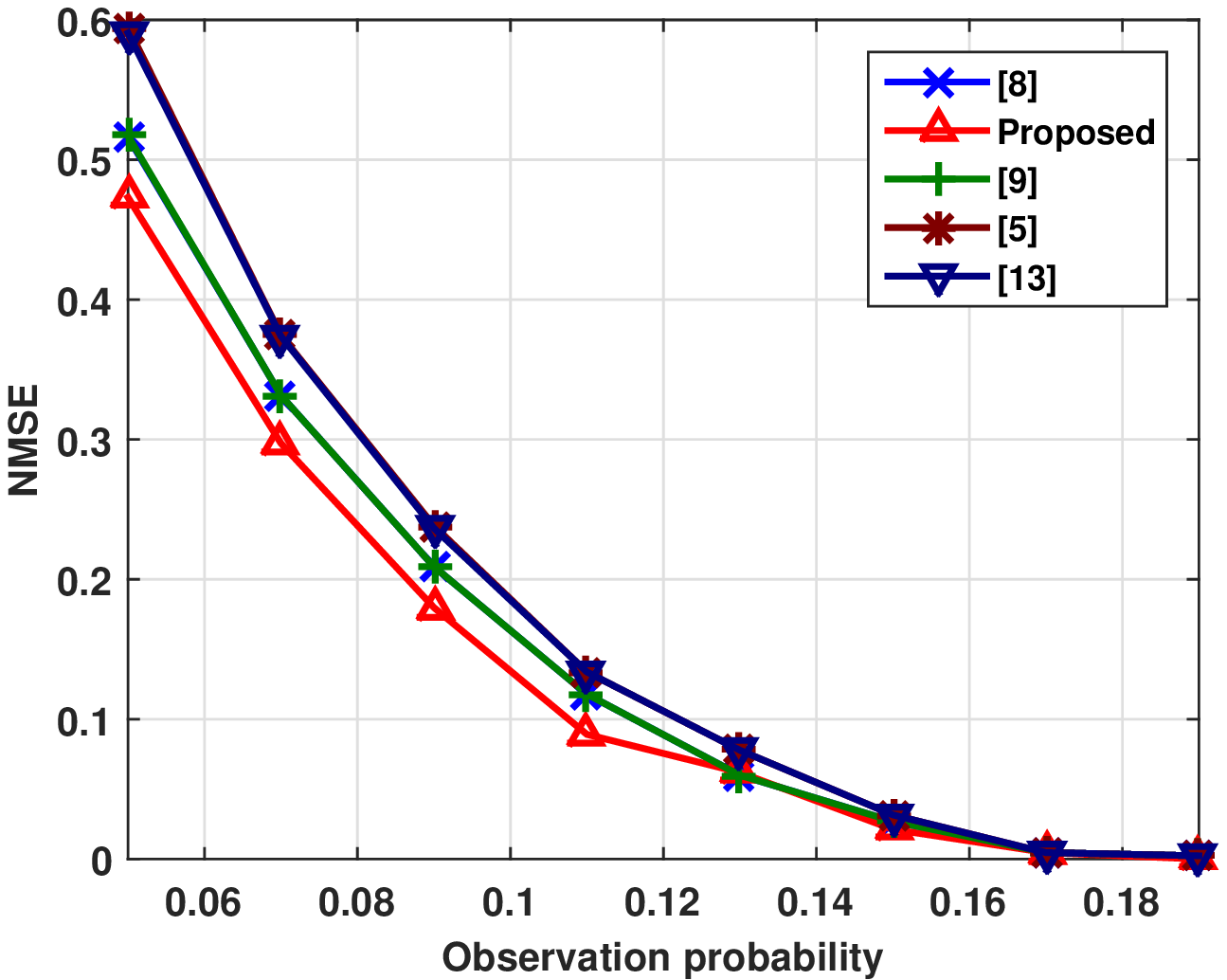}
  \caption{\label{NMSE_pd} NMSE  for matrix completion versus observation probability for $\frac{{\left| {{\Omega _E}} \right|}}{{\left| \Omega  \right|}} = 0.25$.}
\end{minipage}
\hfill
  \centering
\begin{minipage}[b]{0.365\textwidth}
    \includegraphics[width=\textwidth]{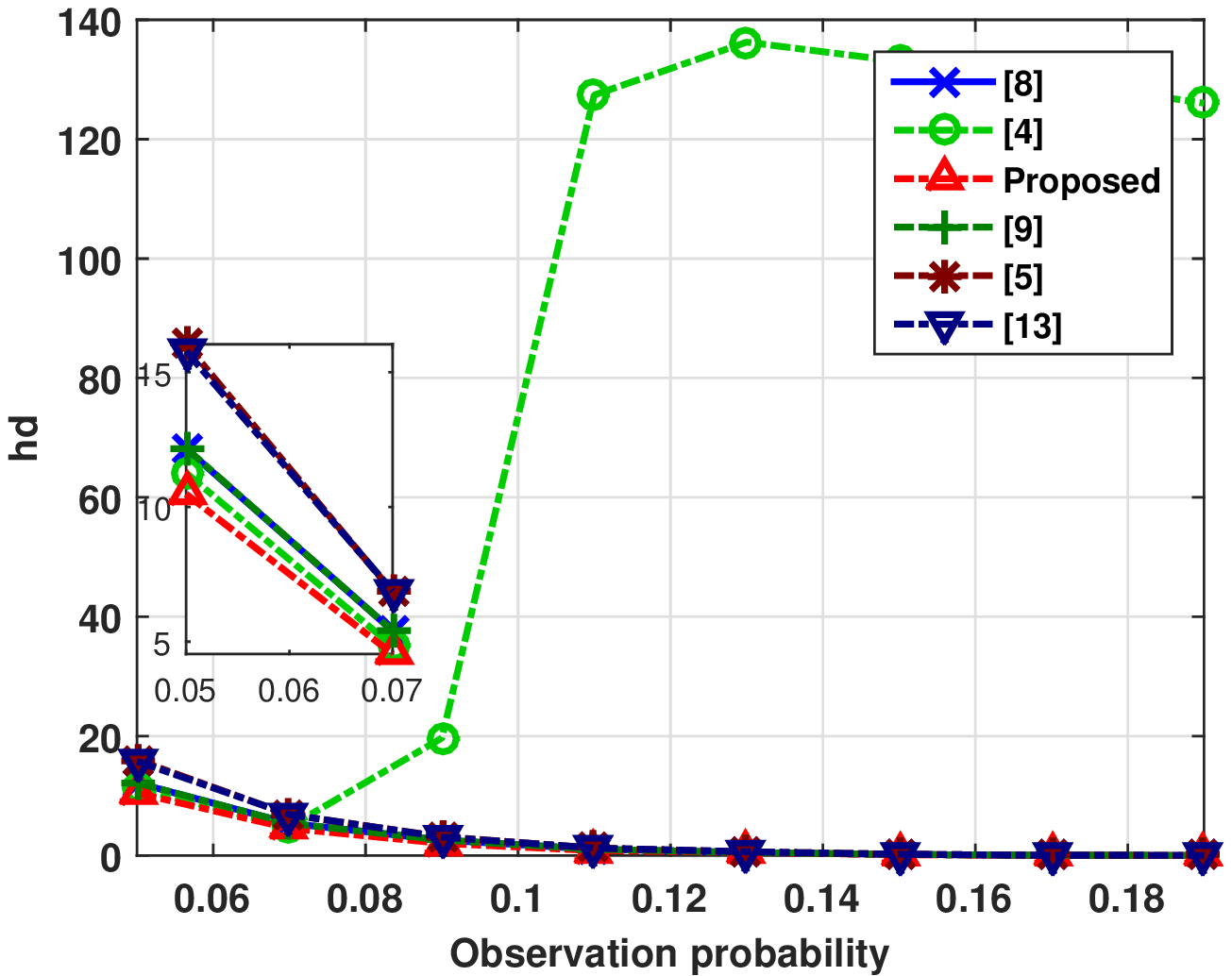}
  \caption{\label{hd_pd} Averaged Hamming distance between the estimated haplotype and original one versus observation probability for $\frac{{\left| {{\Omega _E}} \right|}}{{\left| \Omega  \right|}} = 0.25$.}
  \centering
\end{minipage}
\end{figure}
Simulations are performed with synthetic data. To generate ${\overline{\bf{R}}}_{m\times n}$ in (\ref{Factorization_of_Main_Matrix}), we randomly generate ${\bf{h}}_{n\times1}$ and ${\bf{c}}_{m\times1}$. The random set of observations $\Omega$ is produced while the probability of observation is $\rm{pd}$. Moreover, erroneous observations $\Omega_E$ are defined by changing the sign of some entries observed by $\Omega$, where $\Omega_E\subset\Omega$. Also, we have set $\gamma_1=0.5$,  $\gamma_2=2$, and $p=1.2$. For the initial matrix ${\bf{X}}_0$, we use the rank-one approximation of ${\rm{P}}_\Omega ({\overline{\bf{R}}}_E)$. Simulation results show that this initial matrix satisfies the conditions of Theorem \ref{Convergence}.
To obtain smooth curves, the results are averaged over 50 independent trials, each one with different sets of $\Omega$ and $\Omega_E$.
 Figs. \ref{NMSE_pd} and \ref{hd_pd}  are depicted  for $m=250$, $n=300$, and $0.05\leq $\rm{pd}$\leq 0.2$.  Moreover, ${{\left| \Omega_E  \right|} \mathord{\left/
 {\vphantom {{\left| \Omega  \right|} {\left| {{\Omega _E}} \right|}}} \right.
 \kern-\nulldelimiterspace} {\left| {{\Omega}} \right|}}$ is set to $0.25$, where $\left|\cdot\right|$ denotes the cardinality of a set.  One can observe that the proposed method outperforms the other  ones in estimating haplotypes by generating lower NMSEs and $\rm{hd}$s.
As seen in Fig. \ref{hd_pd}, the MEC based algorithm in \cite{hashemi2018sparse} can only compete with our proposed algorithm for the observation probabilities less than 0.07. This is in accordance with Theorem 2 of \cite{hashemi2018sparse} which states that for a nonzero observation error probability, the increment of observation probability would increase the haplotype estimation error.

\begin{figure}[!ht]
\centering
\begin{minipage}[b]{0.365\textwidth}
    \includegraphics[width=\textwidth]{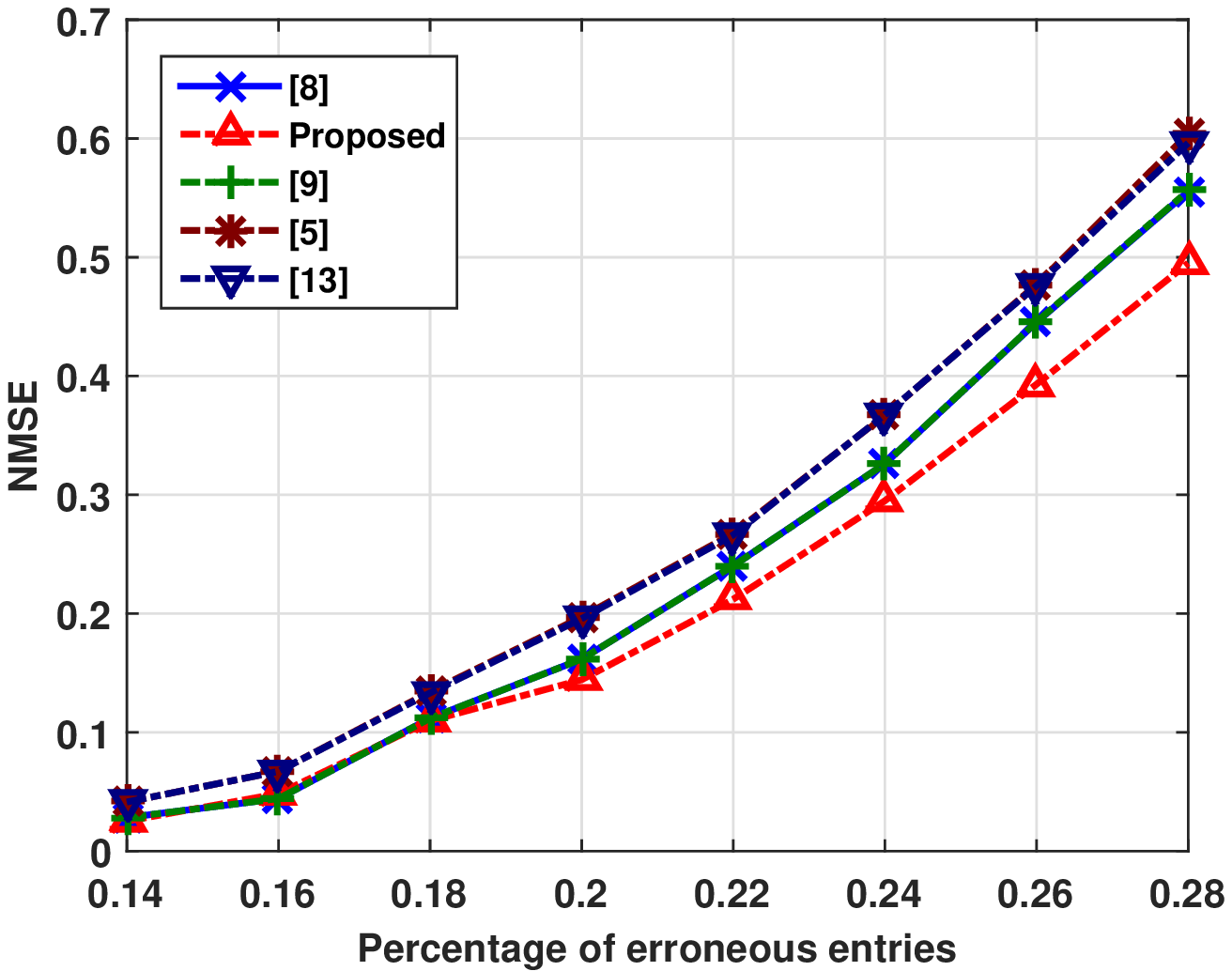}
  \caption{\label{NMSE_Erroneoussamples} NMSE  for matrix completion versus $\frac{{\left| {{\Omega _E}} \right|}}{{\left| \Omega  \right|}}$ for $\rm{pd}=0.07$.}
\end{minipage}
\hfill
  \centering
\begin{minipage}[b]{0.365\textwidth}
    \includegraphics[width=\textwidth]{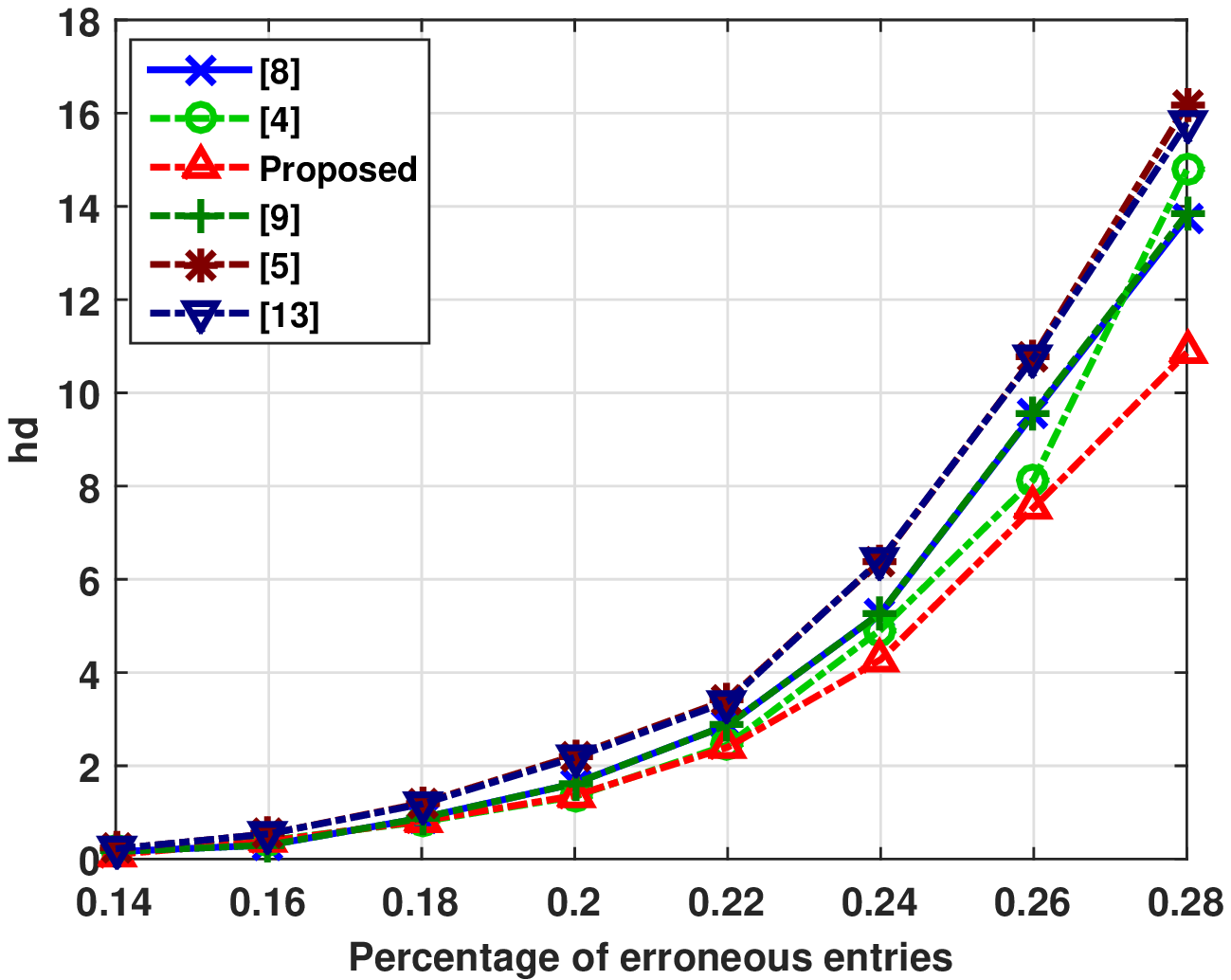}
  \caption{\label{hd_Erroneoussamples} Averaged Hamming distance between estimated haplotype and original one versus $\frac{{\left| {{\Omega _E}} \right|}}{{\left| \Omega  \right|}}$ for $\rm{pd}=0.07$.}
  \centering
\end{minipage}
\end{figure}
 To generate Figs. \ref{NMSE_Erroneoussamples} and \ref{hd_Erroneoussamples}, we change ${{\left| \Omega_E  \right|} \mathord{\left/
 {\vphantom {{\left| \Omega  \right|} {\left| {{\Omega}} \right|}}} \right.
 \kern-\nulldelimiterspace} {\left| {{\Omega}} \right|}}$ from $0.14$ to $0.28$ and proportionally  $p$ from $1.05$ to $1.2$, and set $\rm{pd}=0.07$.
Once more, these results demonstrate the outperformance of the proposed method.
As mentioned in Theorem 2, ${\gamma _1} < {2 \mathord{\left/
 {\vphantom {2 \pi }} \right. \kern-\nulldelimiterspace} \pi }$  is required for our convergence criterion. Moreover, both ${\gamma _1}$  and ${\gamma _2}$  can be selected arbitrarily, so that the function  ${\gamma _1}{\tan ^{ - 1}}\left( {{\gamma _2}{\bf{X}}} \right)$ can satisfyingly approximate ${\rm{sign}}\left( {\bf{X}} \right)$. However, based on our simulation results, for a large value of  ${\gamma _2}$  the runtime increases significantly, and  for a very small value of ${\gamma _2}$,  we will lose the haplotype estimation accuracy.
\section{Conclusion}\label{section.conclusion}
The haplotype assembly problem was investigated. A new matrix completion minimization problem  was proposed which benefits from an error correction mechanism over the manifold of rank-one matrices. The convergence of an iterative algorithm for this problem was theoretically proved. Simulation results demonstrated the validation of the proposed optimization problem by generating more accurate haplotype estimates compared to some recent related methods.



\bibliographystyle{ieeetr}
\bibliography{Hap_SPL_MMM_Manifold_97_08_06_Refrence}

\end{document}